\documentclass[11pt]{amsart}
\usepackage{graphicx}
\newtheorem{thm}{Theorem}[section]
\newtheorem{cor}[thm]{Corollary}
\newtheorem{lem}[thm]{Lemma}

\theoremstyle{definition}

\theoremstyle{remark}
\newtheorem{rem}[thm]{Remark}
\numberwithin{equation}{section}

\def\bc{{\mathbb C}}

\def\bm{{\mathbb M}}
\def\bn{{\mathbb N}}

\def\br{{\mathbb R}}

\def\a{\alpha}
\def\b{\beta}
\def\g{\gamma}  
  \def\D{\Delta}

\def\l{\lambda} 

\def\m{\mu}

\def\s{\sigma} 
\def\t{\tau}
\def\f{\varphi}

\def\tr{\mathop{\rm Tr}}
\def\id{{\bf 1}\!\!{\rm I}}
\def\fb{{\mathbf{f}}}

\def\bx{{\mathbf{b}}}

\def\D{\Delta}

\def\wb{{\mathbf{w}}}
\def\rb{{\mathbf{r}}}
\def\o{\otimes}

\def\a{\alpha}

\begin{document}
\title[Kadison-Schwartz operators]
{On characterizations of bistochastic Kadison-Schwarz operators on
$M_2(\mathbb{C})$}

\author{Farrukh Mukhamedov}
\address{Farrukh Mukhamedov\\
 Department of Computational \& Theoretical Sciences\\
Faculty of Science, International Islamic University Malaysia\\
P.O. Box, 141, 25710, Kuantan\\
Pahang, Malaysia} \email{{\tt far75m@yandex.ru}, {\tt
farrukh\_m@iium.edu.my}}

\author{Hasan Ak\i n}
\address{Hasan Ak\i n, Department of Mathematics, Faculty of Education,
 Zirve University, Kizilhisar Campus, Gaziantep, TR27260, Turkey}
\email{{\tt hasanakin69@gmail.com}}

\begin{abstract}

In this paper we describe bistochastic Kadison-Schawrz operators
acting on $M_2(\mathbb{C})$. Such a description allows us to find
positive, but not Kadison-Schwarz operators. Moreover, by means of
that characterization we construct Kadison-Schawrz operators,
which are not completely positive.

 \vskip 0.3cm \noindent {\it Mathematics Subject
Classification}: 47L07; 46L30; 47C15; 15A48; 81P68
60J99.\\
{\it Key words}: bistochastic mapping, Kadison-Schwarz operator,
complete positive.
\end{abstract}

\maketitle

\section{Introduction}

It is known that entanglement is one of the essential features of
quantum physics and is fundamental in modern quantum technologies
\cite{NC}. One of the central problems in the entanglement theory
is the discrimination between separable and entangled states.
There are several tools which can be used for this purpose. There
are many papers devoted to find a given state is separable (see
\cite{HH}). The most general approach to characterize quantum
entanglement uses a notion of an entanglement witness
\cite{HH1,Cr1,Ter}. One of the big advantages of entanglement
witness is that they provide an economic method of detection which
does not need the full information about the quantum state (see
for recent review \cite{CS1}). Interestingly, the entanglement
witnesses are deeply connected to a theory of positive maps in
operator algebras \cite{Cr,CS,Ha}. Therefore, it would interesting
to find some conditions for the positivity of given mappings. In
this direction there are several papers
\cite{BZ,Cr,Cr2,CS,Ha,MM1,RSW,RSC}. Therefore, it would
interesting to find some conditions for the positivity of given
mappings (see \cite{MM1}-\cite{MT}). In the literature the most
tractable maps, the completely positive mapping, have proved to be
of great importance in the study of quantum system (see
\cite{Choi,OP,OV,Paul,St}). It is therefore of interest to study
conditions stronger than positivity, but weaker than complete
positivity.
 Such a
condition is called {\it Kadison-Schwarz (KS) property}. Note that
KS-operators no need to be completely positive. In \cite{Rob}
relations between $n$-positivity of a map $\phi$ and the KS
property of certain map is established (see also \cite{Bh}). Some
ergodic properties of the Kadison-Schwarz maps were investigated
in \cite{LMM,Gr2,Rob1}. Unfortunately, like completely positive
maps, the description of Kadison-Schwarz maps is not provided.
Very recently, one of the authors of this paper in \cite{MA} has
described bistochastic KS-operators from $M_2(\mathbb{C})$ to
itself. But, in general, the problem still remains open.

In \cite{GM1998} it was proposed to study positive operators $P$
from a von Neumann algebra $M$ to its tensor square $M\otimes M$
(we refer a reader to \cite{GMR,GMb} for recent review on
quadratic operators). It turns out that this kind of mappings have
some applications to quantum information theory. One of such an
application is to detect entangled states. For example, let $P$ be
a block positive, then a state $\phi$ on the algebra $M\o M$ is
separable, then the state $P_*\phi$ is positive. If $\phi$ is
entangled, then $P_*\phi$ may not be positive. This observation
leads to more investigation of operators from $M$ to $M\o M$. In
general, description of this kind of mappings was fully not
studied yet. Some positivity conditions were found in
\cite{MM1,Ma3}. In \cite{MAopen,M2015} it was considered trace
preserving mappings from $M_2(\bc)$ to $M_2(\bc)\otimes M_2(\bc)$,
and each such kind of mappings can be written as a sum of two
"linear" and "nonlinear" operators (see
\eqref{D3-0}-\eqref{D3-2}). In \cite{MA1} mappings of the form
\eqref{D3-1} have been studied. Namely, some sufficient conditions
for positivity  (resp. Kadison-Schwarz property) of the mentioned
mappings were found.

In the present paper we are going to describe or characterize
operators of the form \eqref{D3-2}. To do it, we first in Section
3 we provide a characterization of KS-operators form $M_2(\bc)$ to
$M_2(\bc)$ which improves the main result of \cite{MA}. In section
4, we give a sufficient condition for a class of bistochastic
mappings from $M_2(\bc)$ to $M_2(\bc)\otimes M_2(\bc)$ to be
KS-operator. Note that this class of operators are totally
different from the operators studied in \cite{MA1}. Such a
description allows us to find positive, but not Kadison-Schwarz
operators. Moreover, by means of that conditions one can construct
KS-operators, which are not completely positive. Note that some
parts of this section have been announced in \cite{MA2}. Moreover,
our results allow to produce higher dimensional examples of
positive, but completely positive maps. The proposed approach can
be extended to a more general setting rather that $M_2(\bc)$, and
will produce non trivial examples of positive mappings.

\section{Preliminaries}

In this section we recall some definitions and notations.

 Let
$M_n(\bc)$ be the algebra of $n\times n$ matrices over the complex
field $\bc$. Recall that a linear mapping $\Phi: M_n(\bc)\to
M_m(\bc)$ is called
\begin{enumerate}
\item[(i)] {\it positive} if  $\Phi(x)\geq0$ whenever $x\geq0$;

\item[(ii)] {\it unital} if   $\Phi(\id)=\id$;

\item[(iii)] {\it trace preserving} if $\t(\Phi(x))=\t(x)$, where
$\t$ is the normalized trace;

\item[(iv)] {\it bistochastic} if $\Phi$ is unital and trace
preserving;

 \item[(v)] {\it $n$-positive} if the mapping $\Phi_n: M_n(A)\to
M_n(B)$ defined by $\Phi_n(a_{ij})=(\Phi(a_{ij}))$ is positive.
Here $M_n(A)$ denotes the algebra of $n \times n$ matrices with
$A$-valued entries;

\item[(vi)] {\it completely positive} if it is $n$-positive for
all $n\in\bn$;

\item[(vii)] {\it Kadison-Schwarz operator (KS-operator)}, if one
has
\begin{eqnarray}\label{ks2}
\Phi(x)^*\Phi(x)\leq \Phi(x^*x) \ \ \textrm{for all} \ \ x\in A.
\end{eqnarray}
\end{enumerate}

It is clear that any KS-operator is positive. Note that every
unital 2- positive map is KS-operator, and a famous result of
Kadison states that any positive unital map satisfies the
inequality \eqref{ks2} for all self-adjoint elements $x\in A$.

By $\mathcal{KS}(M_n,M_m)$ we denote the set of all KS-operators
mapping from $M_n(\bc)$ to $M_m(\bc)$.

\begin{thm}\cite{MA}\label{ks-s} The following assertions hold true:
\begin{enumerate}
\item[(i)] Let $\Phi,\Psi\in \mathcal{KS}(M_n,M_m)$, then for any
$\l\in[0,1]$ the mapping $\Gamma=\l \Phi+(1-\l)\Psi$ belongs to
$\mathcal{KS}(M_n,M_m)$. This means $\mathcal{KS}(M_n,M_m)$ is
convex; \item[(ii)] Let $U,V$ be unitaries in $M_n(\bc)$ and
$M_m(\bc)$, respectively, then for any $\Phi\in
\mathcal{KS}(M_n,M_m)$ the mapping $\Psi_{U,V}(x)=U\Phi(VxV^*)U^*$
belongs to $\mathcal{KS}(M_n,M_m)$.
\end{enumerate}
\end{thm}

 By $M_2(\bc)\o M_2(\bc)$ we
mean tensor product of $M_2(\bc)$ into itself. We note that such a
product can be considered as an algebra of $4\times 4$ matrices
$M_4(\bc)$ over $\bc$. By $S(M_2(\bc))$ we denote the set of all
states (i.e. linear positive functionals which take value 1 at
$\id$) defined on $M_2(\bc)$.

Recall that a linear operator $\D: M_2(\bc)\to M_2(\bc)\o
M_2(\bc)$ is said to be  {\it quantum quadratic operator (q.q.o.)}
if it is unital and positive.

A state  $h\in S(M_2(\bc))$ is called {\it a Haar state} for a
q.q.o. $\D$ if for every $x\in M_2(\bc)$ one has
\begin{equation}\label{Haar}
(h\o id)\circ \D(x)=(id\o h)\circ\D(x)=h(x)\id.
\end{equation}

\begin{rem} Let $U:M_2(\bc)\o M_2(\bc)\to M_2(\bc)\o M_2(\bc)$
be a linear operator such that $U(x\o y)=y\o x$ for all $x,y\in
M_2(\bc)$. If a q.q.o. $\D$ satisfies $U\D=\D$, then $\D$ is
called a {\it quantum quadratic stochastic operator} or {\it
symmetric q.q.o}. Recent reviews on this kind of operators can be
found in \cite{GMR,GMb}).
\end{rem}

Recall \cite{BR} that the identity and Pauli matrices $\{ \id,
\sigma_1, \sigma_2, \sigma_3 \}$ form a basis for $M_2(\bc)$,
where
\begin{eqnarray*}
\sigma_1 = \left( \begin{array}{cc} 0 & 1 \\ 1 & 0 \end{array}
\right)~~ \sigma_2 = \left( \begin{array}{cc} 0 & -i \\ i & 0
\end{array} \right)~~ \sigma_3 = \left( \begin{array}{cc} 1 & 0 \\
0 & -1 \end{array} \right).
\end{eqnarray*}

In this basis every matrix $x\in M_2(\bc)$ can  be written as $x =
w_0\id + \wb{\bf \sigma}$ with $w_0\in\bc$, $\wb =(w_1,w_2,w_3)\in
\bc^3$, here $\wb\s=w_1\s_1+w_2\s_2+w_3\s_3$.

\begin{lem}\label{m2}\cite{RSW} The following assertions hold true:
\begin{enumerate}
\item[(a)] $x$ is self-adjoint iff  $w_0,\wb$  are reals;
\item[(b)] $\tr(x) = 1$ iff $w_0 =0.5$, here $\tr$ is the trace of
a matrix $x$; \item[(c)] $x
> 0$ iff $\|\wb\|\leq w_0$, where
$\|\wb\|=\sqrt{|w_1|^2+|w_2|^2+|w_3|^2}$.
\end{enumerate}
\end{lem}

Note that any state $\f\in S(M_2(\bc))$ can be represented by
\begin{equation}\label{state}
{\f}(w_0\id + \wb\sigma)=w_0+\langle\wb,{\mathbf{f}}\rangle, \ \
\end{equation}
where ${\mathbf{f}}=(f_1,f_2,f_3)\in\br^3$ with
$\|{\mathbf{f}}\|\leq 1$. Here as before
$\langle\cdot,\cdot\rangle$ stands for the scalar product in
$\bc^3$. Therefore, in the sequel we will identify a state
$\varphi$ with a vector $\fb\in \br^3$.

In what follows by $\t$ we denote a normalized trace, i.e.
$\t(x)=\frac{1}{2}\tr(x)$, $x\in M_2(\bc)$,

Let  $\Delta:M_2(\bc)\rightarrow M_2(\bc) \otimes M_2(\bc)$ be a
q.q.o.  We write the operator $\Delta $ in terms of a basis in
$\bm_2(\bc)\o\bm_2(\bc)$ formed by the Pauli matrices. Namely,
\begin{eqnarray}\label{ddd0}
&& \Delta \id=\id\otimes \id; \\\label{ddd1} && \Delta
(\sigma_i)=b_i(\id\otimes \id)+\overset{3}{\underset{j=1}{\sum
    }}b_{ij}^{(1)}(\id\otimes \sigma_j)+\overset{3}{\underset{j=1}{\sum
    }}b_{ij}^{(2)}(\sigma_j \otimes \id)+\overset{3}{\underset{m,l=1}{\sum
    }}b_{ml,i}(\sigma_m \otimes \sigma_l),
\end{eqnarray}
where $i=1,2,3$.

In general, a description of positive operators is one of the main
problems of quantum information. In the literature most tractable
maps are positive and trace-preserving ones, since such maps arise
naturally in quantum information theory (see \cite{KR,K,NC,RSW}).
Therefore, in the sequel we shall restrict ourselves to the trace
preserving q.q.o. Hence, from \eqref{ddd0},\eqref{ddd1} one finds

\begin{equation}\label{D3}
\D(x)=w_0\id\otimes\id+
\mathbf{B}^{(1)}\wb\cdot\s\otimes\id+\id\otimes\mathbf{B}^{(2)}\wb\cdot\s+
\sum_{m,l=1}^3\langle\bx_{ml},\overline{\wb}\rangle\sigma_m\otimes\sigma_l,
\end{equation}
where $x=w_0+\wb\s$, $\bx_{ml}=(b_{ml,1},b_{ml,2},b_{ml,3})$, and
$\mathbf{B}^{(k)}=(b_{ij}^{(k)})_{i,j=1}^3$, $k=1,2$.

In general, to find some conditions for $\D$ to be KS-operator, is
a tricky job. Therefore, one can rewrite \eqref{D3} as follows
\begin{equation}\label{D3-0}
\D(x)=\l\D_1(x)+(1-\l)\D_2(x),
\end{equation}
where
\begin{eqnarray}\label{D3-1}
&& \D_1(x)=w_0\id\otimes\id+
\frac{1}{\l}\sum_{m,l=1}^3\langle\bx_{ml},\overline{\wb}\rangle\sigma_m\otimes\sigma_l,
\\[2mm]
\label{D3-2} && \D_2(x)=w_0\id\otimes\id+
\frac{1}{1-\l}\bigg(\mathbf{B}^{(1)}\wb\cdot\s\otimes\id+\id\otimes\mathbf{B}^{(2)}\wb\cdot\s\bigg).
\end{eqnarray}

In \cite{MA1,MAA} we have studied q.q.o. of the form \eqref{D3-1}.
It is found necessary conditions for \eqref{D3-1} kind of
operators to be KS-operator. But operators of the form
\eqref{D3-2} has not been studied yet. Therefore, main aim of this
paper to find some conditions on operators \eqref{D3-2} to be
Kadison-Schwarz. Then using Theorem \ref{ks-s} and our findings
with the results of \cite{MAA}, we can find sufficient conditions
for \eqref{D3-0} to be KS-operator.

\section{Kadison-Schwarz operators from $M_2(\bc)$ to $M_2(\bc)$}

To investigate operators of the form \eqref{D3-2} (see section 4) we
need some preliminary facts from \cite{MA}. In this section we
collect some of them, and improve a main result of \cite{MA}.

It is known that every $\Phi: M_2(\mathbb{C})\rightarrow
M_2(\mathbb{C})$ linear mapping can also be represented in this
basis by a unique $4\times4$ matrix $\textbf{F}$. It is trace
preserving if and only if $\textbf{F}=\left(
                                                                                    \begin{array}{cc}
                                                                                      1 & \mathbf{0} \\
                                                                                      \mathbf{t} & T \\
                                                                                    \end{array}
                                                                                  \right)$
where T is a $3\times3$ matrix and $\mathbf{0}$ and $\mathbf{t}$
are row and column vectors, respectively, so that
\begin{equation}\label{ks3}
\Phi(w_0\id+\wb\cdot\s)=w_0\id+(w_0\mathbf{t}+T\wb)\cdot\s.
\end{equation}

When $\Phi$ is also positive then it maps the subspace of
self-adjoint matrices of $M_2(\mathbb{C})$ into itself, which
implies that $T$ is real. A linear mapping $\Phi$ is unital  if
and only if $t=0$. So, in this case we have
\begin{equation}\label{ks4}
\Phi(w_0\id+\wb\cdot\s)=w_0\id+(T\wb)\cdot\s.
\end{equation}

Hence, any bistochastic mapping $\Phi:M_2(\mathbb{C})\rightarrow
M_2(\mathbb{C})$ has a form \eqref{ks4}. In \cite{MA} it has been
given a characterization bistochastic KS-operators, i.e. the
following

\begin{thm}\cite{MA}\label{ks-d}
Any bistochastic mapping $\Phi: {M}_2(\mathbb{C})\rightarrow
M_2(\mathbb{C})$ is  KS-operator if and only if one has
\begin{eqnarray}\label{ks5}
&&\|T\wb\|\leq\|\wb\|, \ \ T\overline{\wb}=\overline{T\wb} \\
\label{ks6}
&&\bigg\|T[\wb,\overline{\wb}]-\big[T\wb,\overline{T\wb}\big]\bigg\|\leq\|\wb\|^2-\|T\wb\|^2
\end{eqnarray}
for all $\wb\in \bc^3$.
\end{thm}

Let $\Phi$ be a bistochastic KS-operator on $M_2(\bc)$, then it
can be represented by \eqref{ks4}. Following \cite{KR} let us
decompose the matrix $T$ as follows $T=RS$, here $R$ is a rotation
and $S$ is a self-adjoint matrix (see \cite{KR}). Define a mapping
$\Phi_S$ as follows
\begin{equation}\label{F-s}
\Phi_S(w_0\id+\wb\cdot\s)=w_0\id+(S\wb)\cdot\s.
\end{equation}
Every rotation is implemented by a unitary matrix in $M_2(\bc)$,
therefore there is a unitary $U\in M_2(\bc)$ such that
\begin{equation}\label{F-sU}
\Phi(x)=U\Phi_S(x)U^*, \ \ \ x\in M_2(\bc).
\end{equation}

On the other hand, every self-adjoint operator $S$ can be
diagonalized by some unitary operator, i.e. there is a unitary
$V\in M_2(\bc)$ such that $S=VD_{\l_1,\l_2,\l_3}V^*$, where
\begin{eqnarray}\label{ks-D}
D_{\l_1,\l_2,\l_3}=\left(
    \begin{array}{ccc}
      \l_1 & 0 & 0 \\
      0 & \l_2 & 0 \\
      0 & 0 & \l_3 \\
    \end{array}
  \right),
\end{eqnarray}
where $\l_1,\l_2,\l_3\in\br$.

Consequently, the mapping $\Phi$ can be represented by
\begin{equation}\label{F-DU}
\Phi(x)=\tilde U\Phi_{D_{\l_1,\l_2,\l_3}}(x)\tilde U^*, \ \ \ x\in
M_2(\bc)
\end{equation}
for some unitary $\tilde U$. Due to Theorem \ref{ks-s} the mapping
$\Phi_{D_{\l_1,\l_2,\l_3}}$ is also KS-operator. Hence, all
bistochastic KS-operators can be characterized by
$\Phi_{D_{\l_1,\l_2,\l_3}}$ and unitaries. In what follows, for
the sake of shortness by $\Phi_{(\lambda_1,\lambda_2,\lambda_3)}$
we denote the mapping $\Phi_{D_{\l_1,\l_2,\l_3}}$. It is clear to
observe from \eqref{ks5} that  $|\lambda_k|\leq1, k=1,2,3$.\\

In \cite{RSW} it has been given a characterization of completely
positivity of $\Phi_{(\lambda_1,\lambda_2,\lambda_3)}$.

Using Theorem \ref{ks-s} we are going to characterize KS-operators
of the form $\Phi_{(\lambda_1,\lambda_2,\lambda_3)}$.

\begin{thm}\label{KS-m}
$\Phi_{(\lambda_1,\lambda_2,\lambda_3)}$ is a KS-operator if and
only if the following inequalities are satisfied:
\begin{eqnarray}\label{ksiff1}
&&(1+\l_1^2)(3+\l_2^2+\l_3^2-\l_1^2)\leq 4(1+\l_1\l_2\l_3);\\
\label{ksiff2} &&(1+\l_2^2)(3+\l_1^2+\l_3^2-\l_2^2)\leq
4(1+\l_1\l_2\l_3);\\ \label{ksiff3}
&&(1+\l_3^2)(3+\l_1^2+\l_2^2-\l_3^2)\leq 4(1+\l_1\l_2\l_3).
\end{eqnarray}
where $\l_1, \l_2, \l_3\in[-1,1]$.
\end{thm}

\begin{proof}
  'only if' part. Using simple calculation from (3.4) of Theorem \ref{ks-d} with
$T=D_{\lambda_1,\lambda_2,\lambda_3}$ we obtain
\begin{eqnarray}\label{34-1}
A|w_2\overline{w}_3-w_3\overline{w}_2|^2&+&B|w_1\overline{w}_3-w_3\overline{w}_1|^2\nonumber \\[2mm]
&+&C|w_1\overline{w}_2-w_2\overline{w}_1|^2\leq \big(\a|w_1|^2+\b
|w_2|^2+\g |w_3|^2\big)^2,
\end{eqnarray}
where $\wb=(w_1,w_2,w_3)\in \bc^3$ and
\begin{eqnarray}\label{abc}
&&\alpha=|1-\lambda_1^2|,  \ \ \ \beta=|1-\lambda_2^2|, \ \ \ \gamma=|1-\lambda_3^2|\\
\label{abc1} &&A=|\lambda_1-\lambda_2\lambda_3|^2, \ \ \
B=|\lambda_2-\lambda_1\lambda_3|^2, \ \ \
C=|\lambda_3-\lambda_1\lambda_2|^2.
\end{eqnarray}

Due to the inequality $|2\Im(uv)|\leq |u|^2+|v|^2$, one has
\begin{eqnarray}\label{IM1}
|w_i\overline{w}_j-w_j\overline{w}_i|^2=|2\Im(w_iw_j)|^2\leq
|w_i|^4+2|w_i|^2|w_j|^2+|w_j|^4 \ \ (i\neq j)
\end{eqnarray}
Note that this inequality is reachable by appropriate choosing of
values $w_i$ and $w_j$.

Hence, we estimate LHS of \eqref{34-1} by
\begin{eqnarray*}
A(|w_2|^4+2|w_2|^2|w_3|^2+|w_3|^4)+B(|w_1|^4+2|w_1|^2|w_3|^2+|w_3|^4)+C(|w_1|^4+2|w_1|^2|w_2|^2+|w_2|^4)
\end{eqnarray*}
Consequently, from \eqref{34-1} we derive the following one
\begin{eqnarray}
&&|w_1|^4(\alpha^2-B-C)+|w_2|^4(\beta^2-A-C)+|w_3|^4(\gamma^2-A-B)\nonumber \\
\label{ks12}
&&+2|w_1|^2|w_2|^2(\alpha\beta-C)+2|w_1|^2|w_3|^2(\alpha\gamma-B)+2|w_2|^2|w_3|^2(\beta\gamma-A)\geq0
\end{eqnarray}
for all $(w_1,w_2,w_3)\in\bc^3$. It is easy to see that \eqref{ks12}
is satisfied if one has
\begin{eqnarray*}
&&\a^2\geq B+C, \quad  \b^2\geq A+C, \quad  \g^2\geq A+B,\\
&&\a\b\geq C,  \quad  \a\g\geq B,   \quad  \b\g\geq A.
\end{eqnarray*}

Substituting above denotations \eqref{abc},\eqref{abc1} to the last
inequalities, and doing simple
 calculation one derives

\begin{eqnarray}\label{ks17}
&&(1+\l_1^2)(3+\l_2^2+\l_3^2-\l_1^2)\leq 4(1+\l_1\l_2\l_3);\\
\label{ks18} &&(1+\l_2^2)(3+\l_1^2+\l_3^2-\l_2^2)\leq
4(1+\l_1\l_2\l_3);\\ \label{ks19}
&&(1+\l_3^2)(3+\l_1^2+\l_2^2-\l_3^2)\leq 4(1+\l_1\l_2\l_3);\\
\label{ks20} &&\l_1^2+\l_2^2+\l_3^2\leq 1+ 2\l_1\l_2\l_3.
\end{eqnarray}
where $\l_1, \l_2, \l_3\in[-1,1]$.

Now we would like to show that \eqref{ks20} is an extra condition,
i.e. the inequality \eqref{ks20} always satisfies when \eqref{ks17},
\eqref{ks18} and \eqref{ks19} are true. Suppose that
\begin{eqnarray}\label{kseq}
\l_1^2+\l_2^2+\l_3^2= 1+ 2\l_1\l_2\l_3
\end{eqnarray}
is true. We will show that the elements of the surface do not
satisfy the inequalities \eqref{ks17}, \eqref{ks18}  and
\eqref{ks19} except for $(0,0,0), \ (\pm1,\pm1,\pm1)$. Using simple
algebra from \eqref{ks17}, \eqref{ks18} and \eqref{ks19} with
\eqref{kseq} we obtain the followings
\begin{eqnarray*}
  &&(1-\l_1^2)(\l_1^2-\l_1\l_2\l_3)\leq0;\\
  &&(1-\l_2^2)(\l_2^2-\l_1\l_2\l_3)\leq0;\\
  &&(1-\l_3^2)(\l_3^2-\l_1\l_2\l_3)\leq0,
\end{eqnarray*}
where $\l_1, \l_2, \l_3\in[-1,1]$. Due to our assumption
$\l_1\neq\pm1, \l_2\neq\pm1, \l_3\neq\pm1$ from the last
inequalities we infer that
\begin{eqnarray}\label{ksin1}
&&\l_1(\l_1-\l_2\l_3)\leq0\\ \label{ksin2}
&&\l_2(\l_2-\l_1\l_3)\leq0\\ \label{ksin3}
&&\l_3(\l_3-\l_1\l_2)\leq0,
\end{eqnarray}
where $\l_1, \l_2, \l_3\in(-1,1)$. Let $\l_1>0$, then one gets
$\l_1\leq\l_2\l_3.$ It implies $\l_2>0, \l_3>0$ or $\l_2<0, \l_3<0$.
Now assume $\l_2>0, \l_3>0$, then from \eqref{ksin2} and
\eqref{ksin3} one gets
\begin{eqnarray*}
  \l_2\leq\l_1\l_3,\  \ \ \l_3\leq\l_1\l_2.
\end{eqnarray*}
From $\l_1\leq\l_2\l_3$ and $\l_2\leq\l_1\l_3$ one has
$\l_2\leq\l_2\l_3^2$. This means $1\leq\l_3^2$. This contradicts to
our assumption.

Now let $\l_1>0, \l_2<0$ and $\l_3<0$, then from \eqref{ksin2} and
\eqref{ksin3} one finds
\begin{eqnarray}\label{ksin4}
\l_2\geq\l_1\l_3, \ \ \ \ \l_3\geq\l_1\l_2.
\end{eqnarray}
From \eqref{ksin4} one finds $\l_3\geq\l_1^2\l_3$. This implies that
$\l_1^2\geq1$. It is again a contradiction. In case $\l_1<0$, using
the similar argument we will get again contradiction. This implies
the required assertion.

'if' part. Let \eqref{ksiff1}-\eqref{ksiff3} be satisfied. Then it
implies that \eqref{ks20} is always true. This means \eqref{ks12} is
satisfied. This yields \eqref{34-1}, hence
$\Phi_{(\lambda_1,\lambda_2,\lambda_3)}$ is a KS-operator. This
completes the proof.
\end{proof}

Note that the proved theorem provided necessary and sufficient
conditions for the mapping $\Phi_{(\lambda_1,\lambda_2,\lambda_3)}$
 to be KS-operator. In \cite{MA} it was proved only sufficient conditions
 to be KS-operators. Therefore, the last theorem essentially improves a main
  result of \cite{MA}. Moreover, the last theorem allows
us to construct lots of KS-operators, which are not completely
positive.

\section{A class of Kadison-Schwarz operators from $M_2(\bc)$ to $M_2(\bc)\otimes M_2(\bc)$}

In this section we are going to provide description of operators of
the form \eqref{D3-2}. First we need the following auxiliary

\begin{lem}\label{L1}
  Let
  $x=w_0\id\otimes\id+\wb\cdot\s\otimes\id+\id\otimes\rb\cdot\s$. Then
  the following statements hold true:
  \begin{enumerate}
    \item[(i)] $x$ is self-adjoint if and only if $w_0 \in
    \mathbb{R}$ and $\wb, \rb \in \mathbb{R}^3$;
    \item[(ii)] $x$ is positive if and only if $w_0>0$ and $\|\wb\|+\|\rb\|\leq w_0$.
  \end{enumerate}
\end{lem}
\begin{proof} (i). One can see that
    \begin{eqnarray*}
      x^* = \overline{w_0}
      \id\otimes\id+\overline{\wb}\cdot\s\otimes\id +
      \id\otimes\overline{\rb}\cdot\s
    \end{eqnarray*}
So, self adjointness $x$ implies $\overline{w_0}=w_0$,
    $\overline{\wb}=\wb$, $\overline{\rb}=\rb$.

(ii). Let $x$ be self-adjoint. Then from the definition of Pauli
matrices one finds
    \begin{eqnarray*}
      x=\left(
          \begin{array}{cccc}
            w_0+w_3+r_3 & w_1-iw_2 & r_1-ir_2 & 0 \\
            w_1+iw_2 & w_0-w_3+r_3 & 0 & r_1-ir_2 \\
            r_1+ir_2 & 0& w_0+w_3-r_3 & w_1-iw_2 \\
            0 & r_1+ir_2 & w_1+iw_2 & w_0-w_3-r_3 \\
          \end{array}
        \right)
    \end{eqnarray*}
It is easy to calculate that eigenvalues of last matrix are the
    followings
    \begin{eqnarray*}
&&      \l_1=w_0-\|\rb\|+\|\wb\|, \ \
      \l_2=w_0-\|\rb\|-\|\wb\|, \\[2mm]
  &&    \l_3=w_0+\|\rb\|+\|\wb\|, \ \
      \l_4=w_0+\|\rb\|-\|\wb\|
    \end{eqnarray*}

    So, we can conclude that $x$ is positive if and only if the smallest
    eigenvalue is positive. This means $w_0-\|\rb\|-\|\wb\|\geq0$,
    which completes the proof.
\end{proof}

Now we rewrite operator \eqref{D3-2} as $T:M_2(\bc)\to
M_2(\bc)\otimes M_2(\bc)$ given by
\begin{eqnarray}\label{T(x)}
  T(w_0\id+\wb\cdot\s)=w_0\id\otimes\id+\mathbf{A}\wb\cdot\s\otimes\id+\id\otimes\mathbf{C}\wb\cdot\s
\end{eqnarray}
where $\mathbf{A}, \mathbf{C}$ are linear operators on $\bc^3$.

We first find conditions when $T$ is positive. This is given by the
following

\begin{thm}
  The mapping $T$ given by $\eqref{T(x)}$ is positive if and only if
  \begin{eqnarray*}
    \|\mathbf{A}\wb\|+\|\mathbf{C}\wb\|\leq1,
  \end{eqnarray*}
for all $\wb\in\br^3$ with $\|\wb\|=1$.
\end{thm}

\begin{proof}
  Let $x=w_0\id+\wb\cdot\s$ be positive, i.e. $w_0>0,$ $\|\wb\|\leq
  w_0$. Without lost of generality we may assume $w_0=1.$
  Now Lemma $\ref{L1}$ yields that $T(x)$ is positive if and only if
  $\|\mathbf{A}\wb\|+\|\mathbf{C}\wb\|\leq1$. This competes the
  proof.
\end{proof}

\begin{cor}
  Let $\mathbf{A}=\mathbf{C}$ then $T$ is positive if and
  only if $\|\mathbf{A}\|\leq\frac{1}{2}$.
\end{cor}

Now let us turn to the Kadison-Schwarz property.

Define the following mappings
\begin{eqnarray}\label{T(x)1}
  \Phi(x)=w_0\id+2\mathbf{A}\wb\cdot\s
  \end{eqnarray}
\begin{eqnarray}\label{T(x)2}
  \Psi(x)=w_0\id+2\mathbf{C}\wb\cdot\s
\end{eqnarray}
Then one finds
\begin{eqnarray}\label{T(x)3}
  T(x)=\frac{1}{2}\bigg(\Phi(x)\otimes\id+ \id \otimes
  \Psi(x)\bigg).
\end{eqnarray}

\begin{thm}\label{T(x)4a}
  Let $T$ be a mapping given by $\eqref{T(x)3}$. If one has
  \begin{eqnarray}\label{T(x)5}
    \|\wb\|^2-2\|\mathbf{A}\wb\|^2-2\|\mathbf{C}\wb\|^2\geq0
  \end{eqnarray}
  \begin{eqnarray}\label{T(x)6}
    \|\mathbf{A}[\wb,\overline{\wb}]-2[\mathbf{A}\wb,\mathbf{A}\overline{\wb}]\|+\|\mathbf{C}[\wb,\overline{\wb}]-2[\mathbf{C}\wb,\mathbf{C}\overline{\wb}]\|\leq
    \|\wb\|^2-2\|\mathbf{A}\wb\|^2-2\|\mathbf{C}\wb\|^2
  \end{eqnarray}
  Then $T$ is a Kadison-Schwarz operator.
\end{thm}
\begin{proof} From \eqref{T(x)3} one finds that
\begin{eqnarray}\label{T(x)4}
  T(x^*x)-T(x)^*T(x)&=&\frac{1}{2}\bigg(\big(\Phi(x^*x)-\Phi(x)^*\Phi(x)\big)\otimes\id\nonumber\\
  &&+\id\otimes\big(\Psi(x^*x)-\Psi(x)^*\Psi(x)\big)\bigg)\nonumber\\
  &&+\frac{1}{4}\bigg(\id\otimes\Psi(x)-\Phi(x)\otimes\id\bigg)^*\bigg(\id\otimes\Psi(x)-\Phi(x)\otimes\id\bigg).
\end{eqnarray}
Now taking into account the following formula
\begin{equation*}
x^*x=\big(|w_0|^2+\|\wb\|^2\big)\id+\big(w_0\overline{\wb}+\overline{w_0}\wb-i\big[\wb,\overline{\wb}\big]\big)\cdot\s
\end{equation*}
from \eqref{T(x)1} and \eqref{T(x)2} we have
  \begin{eqnarray*}
    \Phi(x^*x)-\Phi(x)^*\Phi(x)=\big(\|\wb\|^2-\|2\mathbf{A}\wb\|^2\big)\id-2i\big(\mathbf{A}[\wb,\overline{\wb}]-2[\mathbf{A}\wb,\mathbf{A}\overline{\wb}]\big)\s,\\
    \Psi(x^*x)-\Psi(x)^*\Psi(x)=\big(\|\wb\|^2-\|2\mathbf{C}\wb\|^2\big)\id-2i\big(\mathbf{C}[\wb,\overline{\wb}]-2[\mathbf{C}\wb,\mathbf{C}\overline{\wb}]\big)\s.
  \end{eqnarray*}
 Therefore, one gets
  \begin{eqnarray*}
    &&\bigg(\Phi(x^*x)-\Phi(x)^*\Phi(x)\bigg)\otimes\id+\id\otimes\bigg(\Psi(x^*x)-\Psi(x)^*\Psi(x)\bigg)\\
   &=&\bigg(\big(\|\wb\|^2-4\|\mathbf{A}\wb\|^2\big)\id-2i\big(\mathbf{A}[\wb,\overline{\wb}]-2[\mathbf{A}\wb,\mathbf{A}\overline{\wb}]\big)\s\bigg)\otimes\id\\
    &&+\id\otimes\bigg(\big(\|\wb\|^2-4\|\mathbf{C}\wb\|^2\big)\id-2i\big(\mathbf{C}[\wb,\overline{\wb}]-2[\mathbf{C}\wb,\mathbf{C}\overline{\wb}]\big)\s\bigg)\\
    &=&\big(2\|\wb\|^2-4\|\mathbf{A}\wb\|^2-4\|\mathbf{C}\wb\|^2\big)\id\otimes\id\\[2mm]
    &&-2i\big(\mathbf{A}[\wb,\overline{\wb}]-2[\mathbf{A}\wb,\mathbf{A}\overline{\wb}]\big)\s\otimes\id-\id\otimes2i\big(\mathbf{C}[\wb,\overline{\wb}]
-2[\mathbf{C}\wb,\mathbf{C}\overline{\wb}]\big)\s
  \end{eqnarray*}
  According to Lemma \ref{L1} we conclude that the last
  expression is positive if and only if \eqref{T(x)5} and
  \eqref{T(x)6} are satisfied.
 Consequently, from \eqref{T(x)4} we infer that under the last
 conditions the mapping $T$ is a KS operator.
  This completes the proof.
\end{proof}

We should stress that the conditions \eqref{T(x)5},\eqref{T(x)6}
are sufficient to be KS-operator.

\begin{cor}\label{T(x)7}
  If the mappings $\Phi$ and $\Psi$ are KS operators, then $T$
  is also KS operator.
\end{cor}

The proof immediately follows from $\eqref{T(x)4}$.

\begin{rem}
  We have to stress that if $T$ is KS operator, then the
  mappings $\Phi$ and $\Psi$ no need to be KS.
\end{rem}

\subsection{Case: $\mathbf{C}=\mathbf{A}$}

Now let us study the operator $T$ given by \eqref{T(x)} when
$\mathbf{C}=\mathbf{A}$. Consequently from \eqref{T(x)} one finds
\begin{eqnarray}\label{T(x)5.1}
  T_A(w_0\id+\wb\cdot\s)=w_0\id\otimes\id+\mathbf{A}\wb\cdot\s\otimes\id+\id\otimes\mathbf{A}\wb\cdot\s.
\end{eqnarray}
From Theorem \ref{T(x)4a} we immediately have the following

\begin{cor}
  Let $T_A$ be a mapping given by \eqref{T(x)5.1}. If one has
\begin{eqnarray*}
    \|\wb\|^2-4\|\mathbf{A}\wb\|^2\geq0
  \end{eqnarray*}
  \begin{eqnarray}\label{T(x)5.2}
   2\|\mathbf{A}[\wb,\overline{\wb}]-2[\mathbf{A}\wb,\mathbf{A}\overline{\wb}]\|\leq
    \|\wb\|^2-4\|\mathbf{A}\wb\|^2.
  \end{eqnarray}
  Then $T_A$ is a Kadison-Schwarz operator.
\end{cor}

Now using the same argument as in section 3, we can write
\begin{equation}\label{T-DU}
T_A(x)=\tilde UT_{D_{\l_1,\l_2,\l_3}}(x)\tilde U^*, \ \ \ x\in
M_2(\bc)
\end{equation}
for some unitary $\tilde U$. Due to Theorem \ref{ks-s} all
bistochastic KS-operators can be characterized by
$T_{D_{\l_1,\l_2,\l_3}}$ and unitaries. In what follows, for the
sake of shortness by $T_{(\lambda_1,\lambda_2,\lambda_3)}$ we denote
the mapping $T_{D_{\l_1,\l_2,\l_3}}$.

 Next we want to characterize KS operators of the form
$T_{(\l_1,\l_2,\l_3)}$.

\begin{thm}
  If
  \begin{eqnarray*}
    &&4(1+8\l_1\l_2\l_3)\geq(1+4\l_1^2)(3+4\l_2^2+4\l_3^2-4\l_1^2),\\
    &&4(1+8\l_1\l_2\l_3)\geq(1+4\l_2^2)(3+4\l_1^2+4\l_3^2-4\l_2^2),\\
    &&4(1+8\l_1\l_2\l_3)\geq(1+4\l_3^2)(3+4\l_1^2+4\l_2^2-4\l_3^2)
  \end{eqnarray*}
  are satisfied, then $T_{(\l_1,\l_2,\l_3)}$ is a KS
  operator.
\end{thm}

\begin{proof} Taking $\mathbf{A}=D_{\l_1,\l_2,\l_3}$ in \eqref{T(x)5.2},
we obtain
\begin{eqnarray}\label{T(x)5.3}
&&4A_1|w_2\overline{w_3}-\overline{w_2}w_3|^2+4A_2|\overline{w_1}w_3-w_1\overline{w_3}|^2+4A_3|w_1\overline{w_2}-\overline{w_1}w_2|^2 \nonumber \\
&&\leq\Big(B_1|w_1|^2+B_2|w_2|^2+B_3|w_3|^2\Big)^2,
\end{eqnarray}
where $\wb=(w_1,w_2,w_3)\in\mathbb{C}^3$ and
\begin{eqnarray}\label{T(x)5.4}
&&A_1=|\l_1-2\l_2\l_3|^2, \ A_2=|\l_2-2\l_1\l_3|^2, \
A_3=|\l_3-2\l_1\l_2|^2, \\ \label{T(x)5.5}
&&B_1=(1-4\l_1^2), \
B_2=(1-4\l_2^2), \ B_3=(1-4\l_3^2).
\end{eqnarray}
By \eqref{IM1} LHS of \eqref{T(x)5.3} can be evaluated as follows
\begin{eqnarray*}
  4A_1\Big(|w_2|^4+2|w_2|^2|w_3|^2+|w_3|^4\Big)&+&4A_2\Big(|w_1|^4+2|w_1|^2|w_3|^2+|w_3|^4\Big)\\
  &+&4A_3\Big(|w_1|^4+2|w_1|^2|w_2|^2+|w_2|^4\Big).
\end{eqnarray*}
Therefore, from \eqref{T(x)5.3} one gets
\begin{eqnarray*}
  &&\Big(B_1^2-4A_2-4A_3\Big)|w_1|^4+\Big(B_2^2-4A_1-4A_3\Big)|w_2|^4+\Big(B_3^2-4A_1-4A_2\Big)|w_3|^4\\
  &&+ 2|w_2|^2|w_3|^2(B_2B_3-4A_1)+2|w_1|^2|w_3|^2(B_1B_3-4A_2)+2|w_1|^2|w_2|^2(B_1B_2-4A_3)\geq0
\end{eqnarray*}
It is obvious that above inequality is satisfied if one has
\begin{eqnarray*}
  &&B_1^2\geq 4A_2+4A_3, \ \  B_2^2\geq 4A_1+4A_3, \ \ B_3^2\geq
  4A_1+4A_2, \\
  &&B_2B_3\geq4A_1, \ \ B_1B_3\geq4A_2, \ \ B_1B_2\geq4A_3.
\end{eqnarray*}
Substituting above denotations \eqref{T(x)5.4}, \eqref{T(x)5.5} to
the last inequalities, and doing some calculations one derives
\begin{eqnarray}\label{T(x)5.6}
    &&4(1+8\l_1\l_2\l_3)\geq(1+4\l_1^2)(3+4\l_2^2+4\l_3^2-4\l_1^2),\\ \label{T(x)5.7}
    &&4(1+8\l_1\l_2\l_3)\geq(1+4\l_2^2)(3+4\l_1^2+4\l_3^2-4\l_2^2),\\ \label{T(x)5.8}
    &&4(1+8\l_1\l_2\l_3)\geq(1+4\l_3^2)(3+4\l_1^2+4\l_2^2-4\l_3^2),\\ \label{T(x)5.9}
    &&1+16\l_1\l_2\l_3\geq 4\l_1^2+4\l_2^2+4\l_3^2,
  \end{eqnarray}
  where $\l_1,\l_2,\l_3\in\Big[-\frac{1}{2},\frac{1}{2}\Big].$

Now using the same argument as in the proof of Theorem \ref{KS-m}
one can show that \eqref{T(x)5.9} is an extra condition. This
completes the proof.
\end{proof}

It is interesting to study when the operator
$T_{(\l_1,\l_2,\l_3)}$ is complete positive. Let us characterize
completely positivity of $T_{(\l_1,\l_2,\l_3)}$.

\begin{thm}\label{CP1}
A map $T_{(\l_1,\l_2,\l_3)}$ is complete positive if and only if
the followings inequalities are satisfied

\begin{itemize}
    \item [(1)]  $|\l_3|<\frac{1}{2};$
  \newline $4\l_1^2+4\l_2^2+4\l_3^2\leq 1+16\l_1\l_2\l_3;$
  \newline $\l_1^2+\l_2^2+\sqrt{\Big(\l_1^2+\l_2^2\Big)^2-4\l_1\l_2\l_3+\l_3^2}\leq\frac{1}{2};$

    \item [(2)]$\l_3=\frac{1}{2}, \ \
    \l_1,\l_2\in\Big[-\frac{1}{2},\frac{1}{2}\Big]$
   \item [(3)]$\l_3=-\frac{1}{2}, \ \ \l_1=\pm\frac{1}{2}, \ \
   \l_2=\mp\frac{1}{2}$
\end{itemize}
\end{thm}
\begin{proof}
  From \cite{Choi} we know that the complete positivity of
  $T_{(\l_1,\l_2,\l_3)}$ is equivalent to the positivity of the
  following matrix
  \begin{eqnarray*}
\widehat{T}_{(\l_1,\l_2,\l_3)}=\left(%
\begin{array}{cc}
  T_{(\l_1,\l_2,\l_3)}(e_{11}) & T_{(\l_1,\l_2,\l_3)}(e_{12}) \\
  T_{(\l_1,\l_2,\l_3)}(e_{21}) & T_{(\l_1,\l_2,\l_3)}(e_{22}) \\
\end{array}%
\right).
\end{eqnarray*}
It is clear that
\begin{eqnarray*}
  &&T_{(\l_1,\l_2,\l_3)}(e_{11})=\frac{1}{2}\left(%
\begin{array}{cccc}
  1+2\l_3 & 0 & 0 & 0 \\
  0 & 1 & 0 & 0 \\
  0 & 0 & 1 & 0 \\
  0 & 0 & 0 & 1-2\l_3 \\
\end{array}%
\right),\\
 &&T_{(\l_1,\l_2,\l_3)}(e_{12})=\frac{1}{2}\left(%
\begin{array}{cccc}
  0 & \l_1+\l_2 & \l_1+\l_2 & 0 \\
  \l_1-\l_2 & 0 & 0 & \l_1+\l_2 \\
  \l_1-\l_2 & 0 & 0 & \l_1+\l_2 \\
  0 & \l_1-\l_2 & \l_1-\l_2 & 0 \\
\end{array}%
\right)
\end{eqnarray*}
and
$T_{(\l_1,\l_2,\l_3)}(e_{22})=\id\o\id-T_{(\l_1,\l_2,\l_3)}(e_{11})$,
$T_{(\l_1,\l_2,\l_3)}(e_{21})=T_{(\l_1,\l_2,\l_3)}(e_{12})^*$.

(1). According to \cite[Theorem 1.3.3]{Bhat} the matrix
$\widehat{T}_{(\l_1,\l_2,\l_3)}$ is positive if and only if
\begin{eqnarray}\label{T(x)5.14}
T_{(\l_1,\l_2,\l_3)}(e_{11})-T_{(\l_1,\l_2,\l_3)}(e_{12})T_{(\l_1,\l_2,\l_3)}(e_{22})^{-1}T_{(\l_1,\l_2,\l_3)}(e_{21})\geq0,
\end{eqnarray}
where $T_{(\l_1,\l_2,\l_3)}(e_{11})$ and
$T_{(\l_1,\l_2,\l_3)}(e_{22})$ are positive matrices.

It is easy to see that $T_{(\l_1,\l_2,\l_3)}(e_{11})$ and
$T_{(\l_1,\l_2,\l_3)}(e_{22})$ are positive if and only if
\begin{eqnarray}\label{T(x)5.15}
  |\l_3|\leq\frac{1}{2}.
\end{eqnarray}

One can calculate that \eqref{T(x)5.14} is equivalent to
\begin{eqnarray*}
  \left(%
\begin{array}{cccc}
  \a_1 & 0 & 0 & \a_4 \\
  0 & 1+\a_3 & \a_3 & 0 \\
  0 & \a_3 & 1+\a_3 & 0 \\
  \a_4 & 0 & 0 & \a_2 \\
\end{array}%
\right)\geq0
\end{eqnarray*}
where
\begin{eqnarray*}
&& \a_1=1+2\l_3-2(\l_1+\l_2)^2, \ \ \
\a_2=1-2\l_3-2(\l_1-\l_2)^2,\\[2mm]
&&\a_3=\frac{(\l_1-\l_2)^2}{2\l_3-1}-\frac{(\l_1+\l_2)^2}{2\l_3+1},\
\ \ \a_4=-2\Big(\l_1^2-\l_2^2\Big).
\end{eqnarray*}

It is known that the  matrix is positive if and only if the
eigenvalues are positive. The eigenvalues of the last matrix can
be calculated as follows
\begin{eqnarray*}
 &&s_1=1, \ \ \
  s_2=\frac{4\l_1^2+4\l_2^2+4\l_3^2-16\l_1\l_2\l_3-1}{4\l_3^2-1},\\
 &&s_3=1-2\l_1^2-2\l_2^2+2\sqrt{\Big(\l_1^2+\l_2^2\Big)^2-4\l_1\l_2\l_3+\l_3^2},\\
  &&s_4=1-2\l_1^2-2\l_2^2-2\sqrt{\Big(\l_1^2+\l_2^2\Big)^2-4\l_1\l_2\l_3+\l_3^2}.
\end{eqnarray*}
To check the their positivity, it is enough to have $s_2\geq0$ and
$s_4\geq0.$ These mean
\begin{eqnarray}\label{T(x)5.16}
&&\l_3\neq\frac{1}{2};\\ \label{T(x)5.17}
&&4\l_1^2+4\l_2^2+4\l_3^2\leq 1+16\l_1\l_2\l_3;\\ \label{T(x)5.18}
&&\l_1^2+\l_2^2+\sqrt{\Big(\l_1^2+\l_2^2\Big)^2-4\l_1\l_2\l_3+\l_3^2}\leq\frac{1}{2};\\
\label{T(x)5.19}
&&\Big(\l_1^2+\l_2^2\Big)^2+\l_3^2\geq4\l_1\l_2\l_3.
\end{eqnarray}
Note that the expression standing inside the square root is always
positive, indeed, we have
\begin{eqnarray*}
  \Big(\l_1^2+\l_2^2\Big)^2+\l_3^2\geq2\Big(\l_1^2+\l_2^2\Big)\l_3\geq2(2\l_1\l_2)\l_3=4\l_1\l_2\l_3.
\end{eqnarray*}
Therefore, from \eqref{T(x)5.15}, \eqref{T(x)5.16},
\eqref{T(x)5.17} and \eqref{T(x)5.18} one has
\begin{eqnarray*}
  &&|\l_3|<\frac{1}{2};\\
  &&4\l_1^2+4\l_2^2+4\l_3^2\leq 1+16\l_1\l_2\l_3;\\
  &&\l_1^2+\l_2^2+\sqrt{\Big(\l_1^2+\l_2^2\Big)^2-4\l_1\l_2\l_3+\l_3^2}\leq\frac{1}{2}.
\end{eqnarray*}

(2). Let $\l_3=\frac{1}{2}$, then $\widehat{T}_{(\l_1,\l_2,\l_3)}$
has the following form

\begin{eqnarray*}
  \widehat{T}_{(\l_1,\l_2,\frac{1}{2})}=\left(%
\begin{array}{cccccccc}
  2 & 0 & 0 & 0 & 0 & \b_1 & \b_1 & 0\\
  0 & 1 & 0 & 0 & \b_2 & 0 & 0 & \b_1 \\
  0 & 0 & 1 & 0 & \b_2 & 0 & 0 & \b_1 \\
  0 & 0 & 0 & 0 & 0 & \b_2 & \b_2 & 0 \\
  0 & \b_2 & \b_2 & 0 & 0 & 0 & 0 & 0 \\
  \b_1 & 0 & 0 & \b_2 & 0 & 1 & 0 & 0 \\
  \b_1 & 0 & 0 & \b_2 & 0 & 0 & 1 & 0 \\
  0 & \b_1 & \b_1 & 0 & 0 & 0 & 0 & 2 \\
\end{array}%
\right),
\end{eqnarray*}
where where $\b_1=\l_1+\l_2$, $\b_2=\l_1-\l_2$,
$\l_1,\l_2\in\Big[-\frac{1}{2},\frac{1}{2}\Big].$ According to the
Silvester's criterion, the matrix given above is positive if and
only if the leading principal minors are positive. Let $D_n,
(n=\overline{1,8})$ be the leading principal minor of
$\widehat{T}_{(\l_1,\l_2,\frac{1}{2})}$. One can see that for each
$n\in\{1,\dots,8\}$, the minor $D_n$ is positive. Hence, if
$\l_3=\frac{1}{2}$ then $\widehat{T}_{(\l_1,\l_2,\frac{1}{2})}$ is
positive.

(3). Now assume $\l_3=-\frac{1}{2}$, then one finds
\begin{eqnarray*}
  \widehat{T}_{(\l_1,\l_2,-\frac{1}{2})}=\left(%
\begin{array}{cccccccc}
  0 & 0 & 0 & 0 & 0 & \b_1 & \b_1 & 0\\
  0 & 1 & 0 & 0 & \b_2 & 0 & 0 & \b_2 \\
  0 & 0 & 1 & 0 & \b_2 & 0 & 0 & \b_1 \\
  0 & 0 & 0 & 2 & 0 & \b_2 & \b_2 & 0 \\
  0 & \b_2 & \b_1 & 0 & 2 & 0 & 0 & 0 \\
  \b_1 & 0 & 0 & \b_2 & 0 & 1 & 0 & 0 \\
  \b_1 & 0 & 0 & \b_1 & 0 & 0 & 1 & 0 \\
  0 & \b_1 & \b_1 & 0 & 0 & 0 & 0 & 0 \\
\end{array}%
\right),
\end{eqnarray*}
where as before  $\b_1=\l_1+\l_2$, $\b_2=\l_1-\l_2$,
$\l_1,\l_2\in\Big[-\frac{1}{2},\frac{1}{2}\Big].$ One can
calculate that principal minors of the last matrix are
\begin{eqnarray*}
  &&D_n=0 \ (n=\overline{1,5}),\\
  &&D_6=(\l_1+\l_2)^2\Big(4(\l_1-\l_2)^2-4\Big),\\
  &&D_7=(\l_1+\l_2)^2\Big(8(\l_1-\l_2)^2-8\Big),\\
  &&D_8=16(\l_1+\l_2)^4,
\end{eqnarray*}
It is easy to see that $\widehat{T}_{(\l_1,\l_2,-\frac{1}{2})}$ is
positive if $D_6\geq0$ and $D_7\geq0$. It implies that
$\l_1=\pm\frac{1}{2}, \
   \l_2=\mp\frac{1}{2}$.
This completes the proof.
\end{proof}

In \cite{RSW} a characterization of completely positivity of
$\Phi_{(\l_1,\l_2,\l_3)}$ has been given. Namely, the following
result holds.
\begin{thm}\label{CP2}
A mapping $\Phi_{(\l_1,\l_2,\l_3)}$ is complete positive if and
only if the following inequalities are satisfied
\begin{eqnarray}\label{T(x)5.20}
&&(\l_1+\l_2)^2\leq(1+\l_3)^2, \\ \label{T(x)5.21}
&&(\l_1-\l_2)^2\leq(1-\l_3)^2, \\ \label{T(x)5.22}
&&\Big(1-\Big(\l_1^2+\l_2^2+\l_3^2\Big)\Big)^2\geq4\Big(\l_1^2\l_2^2+\l_2^2\l_3^2+\l_1^2\l_3^2-2\l_1\l_2\l_3\Big).
\end{eqnarray}
\end{thm}

\textbf{Example.}
  Let us consider a mapping $T_{(a,a,b)}$, where $a,b\in\Big[-\frac{1}{2},\frac{1}{2}\Big]$. Then one can see that
  $\Phi_{(2a,2a,2b)}$ is the corresponding operator. Now let us
  check conditions of  Theorem \ref{CP1} and Theorem \ref{CP2}.
  From conditions of Theorem \ref{CP1} one finds
  \begin{eqnarray}\label{T(x)5.23}
    &&|b|<\frac{1}{2};\\ \label{T(x)5.24}
    &&a^2\leq\frac{1+2b}{8};\\ \label{T(x)5.25}
    &&1-4a^2-2\sqrt{\Big(2a^2-b\Big)^2}\geq0.
  \end{eqnarray}
  Now we would like to show that \eqref{T(x)5.25} is extra
  condition. It means that the left hand side of \eqref{T(x)5.25} is always positive if \eqref{T(x)5.23} and
  \eqref{T(x)5.24} are satisfied. Let \eqref{T(x)5.23} and
  \eqref{T(x)5.24} be true. Then
  \begin{eqnarray*}
    1-4a^2-2\sqrt{\Big(2a^2-b\Big)^2}&\geq&
    1-4\cdot\frac{1+2b}{8}-2\sqrt{\bigg(2\cdot\frac{1+2b}{8}-b\bigg)^2}\\
    &=&1-\frac{1+2b}{2}-2\sqrt{\bigg(\frac{1-2b}{4}\bigg)^2}=0.
  \end{eqnarray*}

  Now from conditions of Theorem \ref{CP2} one has
  \begin{eqnarray}\label{T(x)5.26}
    a^2\leq\frac{(1+2b)^2}{16}.
  \end{eqnarray}

  The graphics of the inequalities \eqref{T(x)5.23},
  \eqref{T(x)5.24} and \eqref{T(x)5.26} are given in the following
figure.

\begin{figure}[h] \centering
\begin{center}
\includegraphics[width=0.55\columnwidth,clip]{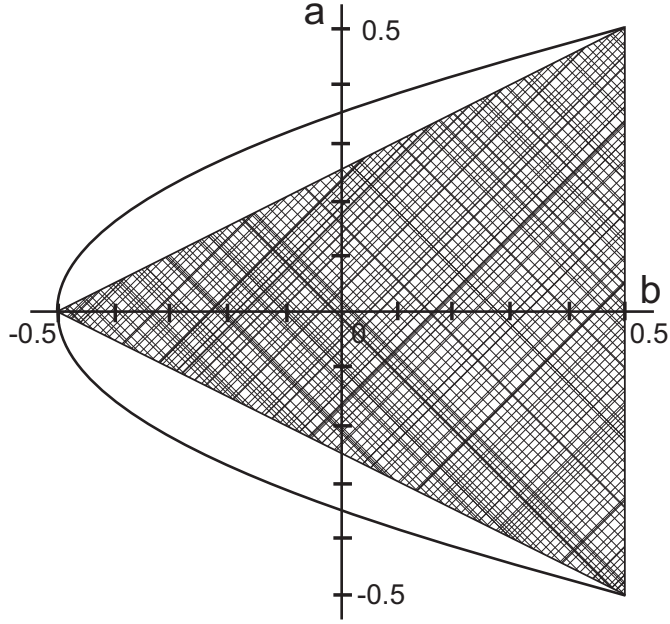}
\caption{Shaded region is CP operators corresponding to
$\Phi_{(2a,2a,2b)}$. White region indicates CP operators
corresponding to $T_{(a,a,b)}$.}
\end{center}
\end{figure}

  From the graph we can see that the class of CP operators
  corresponding to $T_{(a,a,b)}$ are much bigger then the class of
  CP operators corresponding to $\Phi_{(2a,2a,2b)}$

\subsection{Case: $\mathbf{A}\wb=\l\wb$, $\mathbf{C}\wb=\mu\wb$}

In this subsection we consider a more concrete case, namely,
$\mathbf{A}\wb=\l\wb$ and $\mathbf{C}\wb=\mu\wb$. By $T_{\l,\m}$
we denote the corresponding operator (see \eqref{T(x)}). Then one
can see that $\Phi_{(2\l,2\l,2\l)}$ and $\Psi_{(2\m,2\m,2\m)}$ are
the corresponding mappings (see \eqref{T(x)1}-\eqref{T(x)3}).  Due
to Theorem \ref{ks-d} one can find that $\Phi_{(2\l,2\l,2\l)}$ is
a KS-operator if and only if
\begin{eqnarray*}
  2|\l| |1-2\l| \|[\wb,\overline{\wb}]\|\leq \big(1-4\l^2\big)\|\wb\|^2.
\end{eqnarray*}

From $\|[\wb,\overline{\wb}]\|\leq\|\wb\|^2$ (if we choose
$\wb=(0,1,i)$, then one gets $\|[\wb,\overline{\wb}]\|=\|\wb\|^2$)
one finds
\begin{eqnarray*}
  2|\l|(1-2\l)\leq 1-4\l^2.
\end{eqnarray*}
The solution of the last inequality is
$\l\in\big[-\frac{1}{4};\frac{1}{2}\big]$.

Similarly, one finds that $\Psi_{(2\m,2\m,2\m)}$ is a KS-operator
if and only if $\mu \in \big[-\frac{1}{4};\frac{1}{2}\big]$.

From Corollary \ref{T(x)7} we immediately conclude that if $\l,
\mu \in \big[-\frac{1}{4};\frac{1}{2}\big]$ then $T_{\l,\m}$ is a
KS- operator.

Next we want to provide other values of $\l$ and $\mu$ for which
$T_{\l,\m}$ is Kadison-Schwarz.

\begin{thm}\label{T(x)4b} Let $T_{\l,\m}:M_2(\mathbb{C})\rightarrow M_2(\mathbb{C})\otimes
M_2(\mathbb{C})$ be given by \eqref{T(x)}. If
\begin{eqnarray*}
 |\l||1-2\l|+|\mu||1-2\mu|\leq1-2\l^2-2\mu^2
\end{eqnarray*}
is satisfied, then the map $T_{\l,\m}$ is KS-operator.
\end{thm}

\begin{proof} From \eqref{T(x)5},\eqref{T(x)6} one has
\begin{eqnarray*}
 && \l^2+\mu^2\leq\frac{1}{2}\\
 && \big(|\l||1-2\l|+|\mu||1-2\mu|\big)\big\|[\wb,\overline{\wb}]\big\|\leq
  \big(1-2\l^2-2\mu^2\big)\|\wb\|^2.
\end{eqnarray*}
From the arbitrariness of $\wb$ with
$\big\|[\wb,\overline{\wb}]\big\|\leq\|\wb\|^2$ we find
\begin{eqnarray*}
|\l||1-2\l|+|\mu||1-2\mu|\leq1-2\l^2-2\mu^2,
\end{eqnarray*}
which is the required assertion.
\end{proof}

From the figure 2, we conclude that if the pair $(\l,\mu)$ belongs
to the outside of the yellow and red regions, then the mappings
$\Phi_{(2\l,2\l,2\l)}$ and $\Psi_{(2\m,2\m,2\m)}$ are not
Kadison-Schwarz, but the mapping $T_{\l,\m}$ is Kadison-Schwarz.

Now we are interested when the operator $T_{\l,\m}$ is complete
positive.

\begin{thm}\label{T(x)8}
Let $T_{\l,\m}:M_2(\mathbb{C})\rightarrow M_2(\mathbb{C})\otimes
M_2(\mathbb{C})$ be given by \eqref{T(x)}. Then $T_{\l,\m}$ is
completely positive if and only if
\begin{eqnarray*}
  &&\l+\m+1-2\sqrt{\l^2-\l\m+\m^2}\geq0\\
  &&\l+\m\leq1
\end{eqnarray*}
\end{thm}

\begin{proof} It is know \cite{Choi} that the complete positivity of
$T_{\l,\m}$ is equivalent to the positivity of the following
matrix
\begin{eqnarray*}
\widehat{T}_{\l,\m}=\left(%
\begin{array}{cc}
  T_{\l,\m}(e_{11}) & T_{\l,\m}(e_{12}) \\
  T_{\l,\m}(e_{21}) & T_{\l,\m}(e_{22}) \\
\end{array}%
\right).
\end{eqnarray*}
One can calculate that
\begin{eqnarray*}
  &&T_{\l,\m}(e_{11})=\frac{1}{2}\left(%
\begin{array}{cccc}
  1+M_1& 0 & 0 & 0 \\
  0 & 1-M_2 & 0 & 0 \\
  0 & 0 & 1+M_2 & 0 \\
  0 & 0 & 0 & 1-M_1 \\
\end{array}%
\right),\\
 &&T_{\l,\m}(e_{12})=\frac{1}{2}\left(%
\begin{array}{cccc}
  0 & 2\l & 2\m & 0 \\
  0 & 0 & 0 & 2\m \\
  0 & 0 & 0 & 2\l \\
  0 & 0 & 0 & 0 \\
\end{array}%
\right)
\end{eqnarray*}
and $T_{\l,\m}(e_{22})=\id\o\id-T_{\l,\m}(e_{11})$,
$T_{\l,\m}(e_{21})=T_{\l,\m}(e_{12})^*$.  Where $M_1=\l+\m $,
$M_2=\l-\m$. Therefore, we obtain

\begin{eqnarray*}
  \widehat{T}_{\l,\m}=\frac{1}{2}
  \left(%
\begin{array}{cccccccc}
  1+M_1 & 0 & 0 & 0 & 0 & 2\l & 2\m & 0 \\
  0 & 1-M_2 & 0 & 0 & 0 & 0 & 0 & 2\m \\
  0 & 0 & 1+M_2 & 0 & 0 & 0 & 0 & 2\l \\
  0 & 0 & 0 & 1-M_1 & 0 & 0 & 0 & 0 \\
  0 & 0 & 0 & 0 & 1-M_1 & 0 & 0 & 0 \\
  2\l & 0 & 0 & 0 & 0 & 1+M_2 & 0 & 0 \\
  2\m & 0 & 0 & 0 & 0 & 0 & 1-M_2 & 0 \\
  0 & 2\m & 2\l & 0 & 0 & 0 & 0 & 1+M_1\\
\end{array}%
\right)
\end{eqnarray*}

One can calculate the the eigenvalues of $\widehat{T}_{\l,\m}$ are
the followings
\begin{eqnarray*}
  &&\l+\m+1+2\sqrt{\l^2-\l\m+\m^2},\\
  &&\l+\m+1-2\sqrt{\l^2-\l\m+\m^2},\\
  &&1-\l-\m.
\end{eqnarray*}

Hence, $\widehat{T}_{\l,\m}$ is positive if and only if the the
eigenvalues are positive, which implies the assertion.
\end{proof}

\begin{figure}[h] \centering
\begin{center}
\includegraphics[width=0.55\columnwidth,clip]{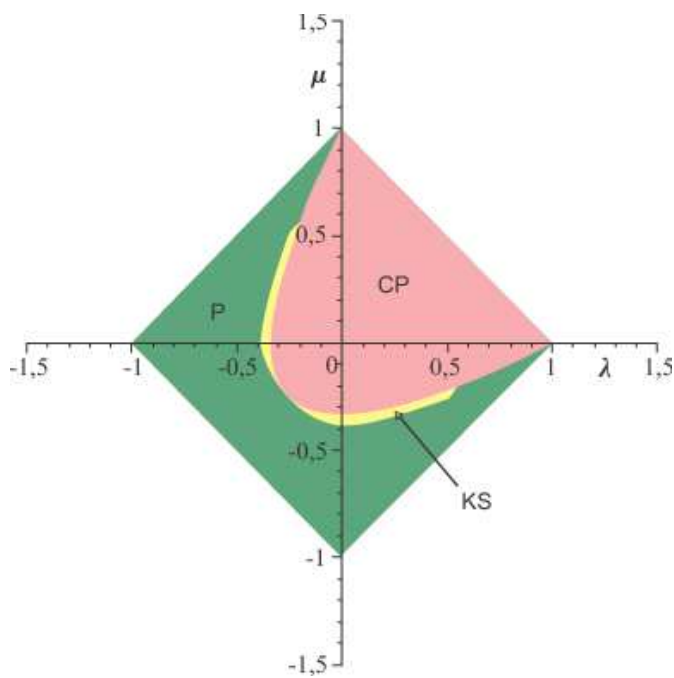}
\caption{If the pair $(\l,\m)$ does not belong to the red region,
then the corresponding mapping $T_{\l,\m}$ is not CP.}
\end{center}
\end{figure}

\section*{Acknowledgement} The first and second named authors (F.M., H.A.) acknowledges the Scientific and
Technological Research Council of Turkey (TUBITAK) for support, and
Zirve University (Gazinatep) for kind hospitality.

\end{document}